\documentclass[contentspage]{article}
\usepackage{amsmath,amsfonts,amssymb,amsthm}
\usepackage{verbatim} 

\title{A Simple Direct Proof of Billingsley's Theorem}
\author{Richard Arratia  \\ Fred Kochman}
\date{July 2012}

\def\e{\mathbb{E} \,}

\def\p{\mathbb{P}}

\def\eps{\varepsilon}
\def\diam{{\rm diam}\,}
\def\vol{{\rm vol}\,}

\newcommand{\ignore}[1]{}

\newtheorem{proposition}{Proposition}

\begin{document}

\maketitle

\begin{abstract} 
Billingsley's theorem (1972) asserts that the
Poisson--Dirichlet process is the limit, as $n \to \infty$, of the
process giving the relative log sizes of the largest prime factor, the
second largest, and so on, of a random integer chosen uniformly from 1
to $n$.  In this paper we give a new proof that directly exploits
Dickman's asymptotic formula for the number of such integers with no
prime factor larger than $n^{1/u}$, namely $\Psi(n,n^{1/u}) \sim n \rho(u)$, to
derive the limiting joint density functions of the 
finite-dimensional projections of the log prime factor processes.  
Our main technical tool is a new criterion for the convergence in distribution of 
non-lattice discrete random variables to continuous random variables.
\end{abstract}

\section{Introduction}\label{intro1}

\subsection{Outline of This Paper}

In this paper, we provide a new proof of Billingsley's theorem
\cite{billingsley72} on the asymptotic joint distribution, as $n \to
\infty$, of the log prime factors of a random integer drawn uniformly
from $1$ to $n$.  Our goal was to stay as close as possible to
straightforward intuition, given Dickman's prior result on the
asymptotic distribution of the largest log prime factor.

Following the description of both the limiting distribution and
Dickman's result, immediately below, we present the heuristic argument
which motivated the present work. The proof itself, which appears in
Section~\ref{simdir}, closely follows the plan of the heuristic and,
in fact, is scarcely longer than that discussion.  This is made
possible by the purely probabilistic technical proposition of
Section~\ref{soft}, which provides a new characterization of
convergence in distribution, applying especially to certain
non-lattice cases.

We conclude with a brief survey of four other published proofs,
including Billingsley's.

\subsection{Review of Billingsley's Theorem}

Billingsley's theorem \cite{billingsley72} describes the joint
distribution of the log sizes of the largest, second largest, and so
on, prime factors of a random integer, by saying that after suitable
normalization, it has a Poisson--Dirichlet limit.  Here are the details.

First, the Poisson--Dirichlet distribution\footnote{with parameter $\theta=1$ } for a random point
$(L_1,L_2,\dots)$ in the infinite-dimensional simplex\footnote{$ \Delta := \left\{ x  \in  \mathbb{R}^\infty: \ x_1,x_2,\dots  
\ge 0, x_1+x_2+\dots=1 \right\}$.\label{footnote Delta}} $\Delta \subset \mathbb{R}^\infty$
can be characterized by specifying the density functions on $\mathbb{R}^k$ induced by  
projecting onto the first $k$ coordinates, for each $k$. These densities
in turn involve the Dickman function\footnote{This is the unique continuous function
on $[0,\infty)$ satisfying the recursion  $\rho(u) = \rho(v) -\int_v^u{\rho(t-1)\frac{dt}{t}}$
for $0\le u-1 \le v \le u$, with initial condition $\rho(u) =1$ on $[0,1]$. See, e.g.  
\cite{Tenenbaum} for more information.}$\rho(u)$.
Specifically, for $k=1,2,\dots$, let
$X = (L_1,L_2,\dots,L_k)$
be the vector giving the first $k$ coordinates of our random
point. 
The distribution of $X$ 
has a density on $\mathbb{R}^k$ given by the formula
\begin{equation}\label{PD fdds}
  f(t_1,\dots,t_k) = \frac{1}{t_1 t_2 \dots t_k} 
\ \rho \left( \frac{1-(t_1+\dots+t_k)}{t_k} \right) 
\end{equation}
on the open set $U$ defined by
$$
U = \{ (t_1,\dots,t_k): \ 
t_1 > \dots >  t_k > 0    \mbox{ and }
t_1 + \dots + t_k < 1 \},
$$ with $f = 0$ outside of $U$.\footnote{There are other useful
  characterizations of the Poisson--Dirichlet distribution that
  de-emphasize the explicit formula~\eqref{PD fdds}: i) The
  Poisson--Dirichlet is \cite{cpc2} the scale invariant Poisson
  process (with intensity $dx/x$ on $(0,1)$), conditional on the sum
  of the arrivals being 1); ii) Ignatov's construction of the
  Poisson--Dirichlet as the ranked list of \emph{spacings} of the
  scale invariant Poisson process.\label{markov footnote}}

Next, 
given $n \ge 1$, pick a random
integer $N$ uniformly from 1 to $n$.  Let $P_i(N)$ be the $i^{\rm th}$
largest prime factor of $N$, with the convention that $P_i(N)=1$ for
$i > \Omega(N)$; here $\Omega(N)$ denotes the
number of prime factors of $N$, including multiplicity.  Let
\begin{equation}\label{def L_i}
  L_i(n) := \log_n P_i(N) = \frac{\log P_i(N)}{\log n},
\end{equation}
where
the random variable $L_i(n)$ is indexed by $n$, the parameter of the 
distribution.  [The random integer $N$, uniformly distributed from 1
  to $n$,
can be recovered via $N = n^{L_1(n) + L_2(n) + \dots}$.]

Billingsley's theorem  then asserts that, as $n \to \infty$,
\begin{equation}\label{billingsley}
(L_1(n),L_2(n),\dots) \Rightarrow (L_1,L_2,\dots),
\end{equation}
where the symbol $\Rightarrow$ denotes convergence in distribution.

The random elements
$(L_1(n),L_2(n),\dots)$ and $(L_1,L_2,\dots)$ lie in the
infinite-\break dimensional space $\mathbb{R}^\infty$.  Since the topology on
this infinite-dimensional product space is characterized by the
continuity of projections onto finitely many factors, standard soft
arguments transform~\eqref{billingsley} into the equivalent statement
that for each fixed $k=1,2,\dots$, random elements of $\mathbb{R}^k$
converge in distribution, with
\begin{equation}\label{billingsley k}
(L_1(n),L_2(n),\dots,L_k(n)) \Rightarrow (L_1,L_2,\dots,L_k).
\end{equation}

\subsection{Heuristic Derivation}

We now give a straightforward heuristic derivation of
\eqref{billingsley k} in which we essentially reduce it to 
the much older result, due to Dickman~\cite{Dickman},
asserting that as $n \to \infty$, for each $t \in (0,\infty)$,
\begin{equation}\label{dickrho}
\frac{1}{n} \Psi\left(n, n^{1/t}\right) = \p(L_1(n) \le 1/t)
\rightarrow \rho(t).
\end{equation}
In \eqref{dickrho}, $\Psi(x,y)$ is, as usual, the number of positive
integers less than or equal to $x$, all of whose prime factors are
less than or equal to $y$.  Using only the monotonicity and continuity
of $\rho(\cdot)$, having \eqref{dickrho} hold for each $t>0$ is
equivalent to having~\eqref{dickrho} \emph{uniformly} over $t$ in
\emph{compact} subsets of $(0,\infty)$.  For a simple derivation
of~\eqref{dickrho}, see \cite[page~365]{Tenenbaum}, or
\cite[p.~492]{T}; this also gives a sharper error term and a broader
region of uniformity, although we do not use these.  The sole analytic
number theory input needed to derive~\eqref{dickrho} is Mertens'
theorem, 1874, which asserts that there exists a constant $c_0$ such
that, as $x \to \infty$,
\begin{equation}
\sum_{p \le x} 1/p \ = c_0 + \log \log x + o(1). \label{mertens}
\end{equation}

Fix a value of $k$.
For $n\ge p_1 \ge p_2 \ge \dots \ge p_k$ let $D(p_1,\dots,p_k)$
be the joint event that $(p_1 \dots p_k) |N$ and that also $N/(p_1 \dots p_k) $ is 
$p_k$-smooth, \emph{i.e.}, has no prime factor larger than $p_k$. For distinct $k$-tuples of primes
the events $D(p_1,\dots,p_k)$ are disjoint.  

The probability $P((p_1 \dots p_k) |N)$ is approximately $1/(p_1 \dots
p_k) $.  Conditional on $(p_1 \dots p_k) |N$, the quotient $N/(p_1
\dots p_k) $ is uniformly distributed over the interval $[1, \lfloor
n/(p_1p_2\dots p_k) \rfloor]$. Therefore by Dickman's
Theorem~\eqref{dickrho}, the conditional probability that $N/(p_1
\dots p_k) $ is $p_k$-smooth becomes approximately
$$
\rho \left( \frac{1-(\tau_1+\dots+\tau_k)}{\tau_k} \right)
$$
where $\tau_i = \log p_i / \log n$, $ i = 1,\dots, k$. Further, by continuity of $\rho$ and the
requirement that $\tau_i \in [t_i, t_i +\Delta t_i]$   we can safely replace each $\tau_i$ 
in the above expression with $t_i$.

For $t_1 > \dots >  t_k > 0    \mbox{ and } t_1 + \dots + t_k < 1$, 
the event $E(t_1, \dots,t_k)$ that
$$
t_i < L_i(n) < t_i + \Delta t_i, i=1,\dots,k
$$
is the union, over all $k$-tuples of primes  $p_1,\dots,p_k$, 
each $p_i \in (n^t_i,n^{t_i+\Delta t_i})$, of the disjoint events $D(p_1,\dots,p_k)$.
Therefore 
$$
P\left(E(t_1, \dots,t_k)\right) \doteq \left(\sum_{p_1}\frac{1}{p_1}\right)\dots \left(\sum_{p_k}\frac{1}{p_k}\right) \rho \left( \frac{1-(t_1+\dots+t_k)}{t_k} \right).
$$
Since $$\sum_{p \in (n^t,n^{t+\Delta t})} 1/p \doteq \log \log n^{t+\Delta t} -
 \log \log n^t =  \mbox{$\log((t+\Delta t)/t)$}  \doteq \Delta t /t$$ for small
$\Delta t,$ we conclude that 
$$
P\left(E(t_1, \dots,t_k)\right) \doteq  \frac{\Delta t_1}{t_1}\dots\frac{\Delta t_k}{t_k}\rho \left( \frac{1-(t_1+\dots+t_k)}{t_k} \right) 
$$
for large $n$ and small $\Delta t_i$. We interpret this as confirming~\eqref{billingsley k}.

While the above argument is only heuristic,
we would be gratified if the reader finds it 
simple, direct, and compelling.  In this paper, our main goal is to supply
a rigorous proof, in the spirit of the above reasoning.  This is facilitated
by our purely probabilistic Proposition \ref{proposition 1},
giving a new criterion for convergence in distribution. 
It is \emph{soft} in the sense that it
gives no handle on the actual magnitude of the error terms. But once
it is in hand, the remaining argument to
prove Billingsley's theorem is two pages long, following the above reasoning
closely, and is given in
Section~\ref{simdir}.

\section{A Soft Result on Weak Convergence}\label{soft}

Our goal in this section is to prove convergence in distribution, of certain kinds of
discrete random variables, to variables possessing density functions.
In typical contexts involving the convergence of a sequence $\{X_n\}$
of discrete variables to a continuous variable $X$, the discrete
elements are supported on lattices of successively finer mesh, and
conditions on the point probabilities are available that ensure such
convergence. But since logs of primes do not live in any kind of
lattice, the standard tools do not apply, and so a new one, such as
Proposition~\ref{proposition 1}, seems necessary.

Our proposition presupposes a limiting density $f$ with certain
continuity properties, but makes provision for discontinuities at the
boundary of its support, such as exhibited by the function $f$ in
\eqref{PD fdds}.  This accounts for the finicky phrasing of the
continuity hypothesis in the lemma.  It may be surprising that issues
such as ``regularity of the boundary'' play no role in our
proposition.

Our proof relies instead on a ``continuity'' property  of probability
measures:  
if events $E_1 \subset E_2
\subset
\dots$ have countable union $E = \cup_{m \ge 1} E_m$, then $\p(E_m)
\to \p(E)$ as $m \to \infty$.

We will use standard notation: 
\begin{itemize}
\item $ || x-y||$ is the Euclidean distance between
two points $x,y \in   \mathbb{R}^k$;  
\item $\diam(B)$ is the diameter of the
set $B$,  \emph{i.e.}, $\diam(B):=\sup_{x,y \in B} || x-y||$; 
\item $d(x,S) := \inf_{y \in S}  || x-y||$ is the distance from a point $x$ to a set
$S$, and 
\item $d(A,B) := \inf_{x \in A} d(x,B)$ is the distance between
sets $A$ and $B$.
\end{itemize}

Our {\it a~priori} characterization of weak convergence will be 
the collection of equivalent conditions in the 
\emph{Portmanteau Theorem}, as presented in 
\cite[p.~16~and~26]{BillingsleyWeak2}. Kallenberg \cite{kallenberg}
attributes this result to A.~D.~Alexandrov, 
\cite{alexandrov}.

Here we quote only the parts that we actually invoke:
\begin{quote}
{\bf Portmanteau Theorem:}
For random elements $X,X_1,X_2,\dots$ of a metric space, the following are equivalent:
\begin{itemize}
\item [\it i)] $X_n \Rightarrow X$, \emph{i.e.}, $X_n$ converges in distribution to
  $X$.

\item [\it ii)] $\e g(X_n) \to \e g(X)$ for all bounded uniformly continuous
  functions $g$.

\item [\it iii)] $\liminf_n \p(X_n \in G) \ge \p(X \in G)$ for all open $G$.
 
\end{itemize}
\end{quote}
Our proposition provides a new necessary and sufficient criterion that refines and
{\em appears} to weaken the requirements of Item {\it iii)}, above, in
certain cases.  

\begin{proposition}\label{proposition 1}
Suppose $X$ is a random element of $\mathbb{R}^k$ with density $f$
of the form  $f = f_U 1_U$, where $U \subset \mathbb{R}^k$ is an open
set, the function $f_U : U \to (0,\infty)$ is continuous, and 
$1_U: \mathbb{R}^k \to \{0,1\}$
denotes
the indicator function of $U$.
Let $X_n$, $n=1,2,\dots$, be arbitrary random elements of $\mathbb{R}^k$.

A necessary and sufficient condition for 
$X_n \Rightarrow X$, as $n \to \infty$, is the following: 

For every $\eps > 0$,  there exists $R < \infty$, such that
every  closed coordinate box $B$ satisfying $B \subset U \mbox{  and }$
\begin{equation}\label{lemma hyp 0}
R \ \diam(B) < d(B, U^c) 
\end{equation}
also satisfies
\begin{equation}\label{lemma hyp 1}
  \liminf_n \p(X_n \in B) \ge  (1-\eps)\  \vol(B)\  \inf_B f.
\end{equation}
\end{proposition}

\begin{proof}
\noindent {\bf Necessity:}  Take $R=0$, so that
  every closed box $B \subset U$ satisfies~\eqref{lemma hyp 0}.   
Assuming $X_n \Rightarrow X$, for
  the closed box $B$ we apply {\it iii)} of the Portmanteau theorem to
  the interior $B^\circ$, to get
 \mbox{$\liminf \p(X_n \in B)$} $\ge \liminf \p(X_n \in B^\circ) \ge \p(X \in
  B^\circ) = \int_B f \ge \vol(B) \ \inf_B f \ge $ \mbox{$(1-\eps)\
  \vol(B)\  \inf_B f.$}

\noindent {\bf Sufficiency:}
Assume 
\eqref{lemma hyp 0} and~\eqref{lemma hyp 1}.
By the Portmanteau theorem, Item {\it iii)}, it suffices to show that
\begin{equation}\label{iv}
  \mbox{ for all open } G \subset \mathbb{R}^k, \ \ 
\liminf_n \p(X_n \in G) \ge \p(X \in G).
\end{equation}
We claim that without loss of generality we may assume $G \subset U$ or, 
equivalently, that $G = G \cap U$.  That is, it suffices to show that
\begin{equation}\label{iv'}
  \mbox{ for all open } G \subset U, \ \ 
\liminf_n \p(X_n \in G) \ge \p(X \in G).
\end{equation}
[To see this, given $G$ open let $H = G \cap U$.  Then since $H$ is
  open 
and $H \subset U$,~\eqref{iv'} implies  
$\liminf \p(X_n \in H ) \ge \p(X \in H)$.
Now  $\p(X_n \in G ) \ge \p(X_n \in H )$,   so that  
$\liminf \p(X_n \in G ) \ge \liminf \p(X_n \in H ) 
    \ge \p(X \in H)$  
$    =  \p(X \in G)$, using $1=\p(X \in U)$.  This shows that
\eqref{iv'} implies~\eqref{iv}.]

Let $\eps >0$ be given, and fix $G$ open, with $G \subset U$.  Fix
an $R>0$
that works with $\eps$ in the condition for~\eqref{lemma hyp 1}.

Since $1 = \p( \cup_m \{ ||X|| \le m \})$, there exists $m_1$ such that 
\begin{equation}\label{1a}
\p( ||X|| > m_1 ) <\eps  .
\end{equation}
Fix such an $m_1$.

The distance $d(x,G^c)$ from $x$ to the closed set $G^c$  is a
continuous function of $x$, and is strictly
positive for $x \in G$. Hence, with  $G_j := \{x \in G: d(x,G^c) >
1/2^j \}$, we have  $G = \cup_j G_j$.  Hence $P(G \setminus G_j) \to
0$.  So there exists
$m_2$ such that for $m \ge m_2$,
\begin{equation}\label{1b}
\p\left(X \in G, \mbox{ and } d(X, G^c) \le  (1+R) \, \sqrt{k} /2^m\right) < \eps.
\end{equation}
Fix such an $m_2$.
(The factor $\sqrt{k}$ is the ratio of diameter to side length for a
cube in $k$ dimensions, and will be used below.)

We  define ``$B$ is a level-$m$ dyadic cubelet'' to mean that $B$ has the form
  $$B = \prod_{j=1}^k \left[i_j/2^m,\left(1+i_j\right)/2^m\right].$$
Since $f_U:U \to (0,\infty)$ is continuous, $\log f_U: U \to
(-\infty,\infty)$ is also continuous, and hence uniformly continuous
on compact subsets of $U$.  The compact set we have in mind is $$K_0
:= \left\{x: d(x,0) \le 1+m_1\ {\rm and\ } d(x, G^c) \ge R \, \sqrt{k}
/2^{m_2}\, \right\}.$$ Hence there exists $m_3$ with $2^{m_3} \ge
\sqrt{k}$ so that, for every level-$m$ dyadic cubelet $B \subset U$
with $m \ge m_3$, if $B \subset K_0$, then
\begin{equation}\label{1c}
     (1- \eps) \sup_B f \le \inf_B f.
\end{equation}
Fix such an $m_3$.

Note that for sets $B$ satisfying~\eqref{1c}, 
since  $\p(X \in B) = (\int_B f dx) \le \vol(B) \sup_B f$,
we have
\begin{equation}\label{1d}
    \vol(B)\inf_B f \ge  \vol(B) (1- \eps)
    \sup_B f \ge  (1- \eps)  P(X \in B)      .
\end{equation}

Now take $m = m_3$.  Let $C$ be the set of level $m$ dyadic cubelets $B$
such that $B$ has nonempty intersection with  the ball of radius
$m_1$ centered at the origin,   and  $d(B, G^c)> R \,  \sqrt{k}/2^m$.
In particular, for $B \in C$,  
\begin{equation}\label{1e0}
B \subset G \subset U,\mbox{  and } 
R \ \diam(B) < d(B,G^c) \le d(B, U^c),
\end{equation}  
\emph{i.e.}, if $B \in C$ then~\eqref{lemma hyp 0} is satisfied.
The condition on intersecting  the ball of radius $m_1$
implies that $C$ is finite.  Let  $K = \cup_C \, B$.  Any point $x \in
U$
lying outside the regions targeted by the events in
\eqref{1a} and~\eqref{1b}, that is,  
with $d(x,0) \le m_1$ and 
$d(x,G^c)>(1+R) \, \sqrt{k} /2^m$,  lies in a cubelet $B \in C$, hence
$x \in K$. Thus
$\p(X \in G \setminus K) <2 \eps$,
by the combination
of~\eqref{1a} and~\eqref{1b}. This shows that
\begin{equation}\label{1e}
\p( X \in K ) > \p( X \in G) - 2 \eps .
\end{equation}

We need a disjoint union for use later in this proof, so let $L =
\cup_C B^\circ$, the union of the interiors of the cubelets whose
union is $K$.  Since $X$ has a density $f$ with respect to Lebesgue
measure, $\p( X \in L ) = \p( X \in K ) > \p( X \in G) - 2 \eps$.  As
a finite union of open sets, $L$ is open.  Define $s_i(B)$, the
level-$i$ shrink of the cubelet $B$, to be the closed cubelet with the
same center and orientation as $B$, with side shrunk by a factor of
$(1-1/2^i)$.  By virtually the same argument as used for~\eqref{1a}
and \eqref{1b}, with $L$ in the role of $G$, there exists an $i_0$
such that for all $i \ge i_0$, for $J_i := \cup_C \, s_i(B)$, $\p( X
\in J_i) > \p( X \in L ) - \eps$.  Fix such an $i_0$ and write $s$ for
$s_{i_0}$.  We now have a finite collection of disjoint closed boxes
$s(B)$, indexed by $C$, such that
$$J := \cup_C \, s(B) \subset G$$  satisfies  
\begin{equation}\label{1f}
  \sum_C \p( X \in s(B) ) = \p( X \in J)  > P( X \in G) - 3 \eps.
\end{equation}

Comparing with~\eqref{1e0}, the shrunken boxes $s(B)$  have 
smaller diameter, and larger distance to $U^c$, so for $B \in C$, the
box $s(B)$ satisfies~\eqref{lemma hyp 0}, hence
$$
\liminf_n P(X_n \in s(B)) \ge (1-\eps)\  \vol(s(B))  \, \inf_{s(B)} f.
$$
Using the finiteness of $C$, there exists $n_1$, for all $n > n_1$, 
for all $B \in C$,
\begin{equation}\label{1g}
    P(X_n \in s(B)) \ge (1-2\eps)\  \vol(s(B))\, \inf_{s(B)} f.
\end{equation}
Fix such a choice of $n_1$.

Note that in~\eqref{1c}, replacing $B$ by $s(B)$ does not increase
the sup on the left, nor does it decrease the inf on the right,
so~\eqref{1c} and hence~\eqref{1d} hold for $s(B)$, so for $B \in C$,
\begin{equation}\label{1h}
    \vol(s(B))\inf_{s(B)} f \ \ge  \   (1- \eps)  \, \p(X \in s(B))      .
\end{equation}

Combining~\eqref{1g} with~\eqref{1h} we have, 
for all $n > n_1$, 
for all $B \in C$,
\begin{equation}\label{1i}
    \p(X_n \in s(B)) \ge (1-3\eps)\  \, \p(X \in s(B)) .
\end{equation}

Finally, we combine~\eqref{1f} and~\eqref{1i},  the finite
disjoint union of closed boxes $J = \cup_C \, s(B) \ \subset G$.  This yields,
for all $n > n_1$,
$$
  \p(X_n \in G) \ge \p(X_n \in J) = \sum_C \p(X_n \in s(B))
    \ge  \sum_C  (1-3 \eps) \p(X \in s(B))
$$
$$
 = (1-3 \eps) \p(X \in J) 
    \ge  \p(X \in G) - 6 \eps.
$$
Since $\eps$ was arbitrarily small, we have proved 
\mbox{$\liminf \p(X_n \in G) \ge \p(X \in G)$}, and so by Item {\it iii} of the Portmanteau theorem
we are done.
\end{proof}

\section{The Simple Direct Proof}\label{simdir}

In this section, we supply the promised proof of Billingsley's theorem
\eqref{billingsley}, under the original hypothesis that the random
integer is picked uniformly from $1$ to $n$.  The only inputs from
number theory are Dickman's statement~\eqref{dickrho}, and Mertens'
theorem,~\eqref{mertens}.

\begin{proof}
For each fixed $k=1,2,\dots$, we prove the weak convergence
expressed in~\eqref{billingsley k} by applying Proposition 
\ref{proposition 1}, where $X_n:= (L_1(n),\dots,L_k(n))$ from the left side of 
\eqref{billingsley k}, 
and $X$ has the density given by~\eqref{PD fdds}.

Our only task is to show that the key hypothesis \eqref{lemma hyp 1}
is satisfied.  We will see that with the choice $R = k/(2 \eps)$ the
uniformity requirement~\eqref{lemma hyp 0} is satisfied.  Fix a closed
coordinate box $B \subset U$, and use the notation
$$
    B = \prod_{i=1}^k \ [t_i, t_i + \Delta t_i].
$$
Since $B \subset U$,
$$
   0 < t_k < t_k+\Delta t_k < t_{k-1} < \dots < t_1 < t_1 + \Delta t_1 < 1.
$$
Also we have 
\begin{equation}\label{prob n}
 \p( X_n \in B ) = \frac{1}{n}  \left|\left\{ m \le n:  P_i(m) \in
   \left[n^{t_i},n^{t_i + \Delta t_i}\right],  \ i=1,\dots, k \ \right\} \right|. 
\end{equation}
Since $p_1 \ge p_2 \ge \dots \ge p_k$ are the $k$ largest prime factors of $m$ if and only if 
$m = p_1\dots p_k l$ for some $p_k$-smooth integer $l$ with $1 \le l \le n/(p_1\dots p_k)$,
collecting all possible $k$-tuples of largest prime factors yields
\begin{equation}\label{sum pi}
  \p( X_n \in B )
 =  \frac{1}{n} \sum_{p_1,\dots,p_k} \Psi( n/(p_1\dots p_k), p_k ),
\end{equation}
where independently for $i=1$ to $k$ we sum over  $p_i$ for which 
$$
t_i \le \log_n p_i \le t_i + \Delta t_i.
$$

Let
$$
\alpha := 1 - \sum_1^k (t_i + \Delta t_i), \ \mbox{ and } 
   u_0 := \frac{1-(t_1+\dots + t_k)}{t_k}.
$$
The condition $B \subset U$ implies that $\alpha > 0$ and 
$u_0 < \infty$.

Every $\Psi$ that occurs in the sum in~\eqref{sum pi} above has  
the form $\Psi(x,y)$, with $x=n/(p_1\dots p_k)$, and $y=p_k$,
so that $x \ge n^\alpha$,  and
$u := \log x / \log y \in (0,u_0]$.  Hence,
Dickman's
estimate on
$\Psi$, given by~\eqref{dickrho}, implies
$$
  \p( X_n \in B ) =  \frac{1}{n} \sum_{p_1,\dots,p_k}
  \frac{n}{p_1\dots p_k} \rho\left( \frac{\log n - \log p_1 - \dots
  - \log p_k}{\log p_k} \right) \ (1+o(1)).
$$
In the sum above, the smallest value of the function $\rho$, corresponding to the largest argument, 
occurs when the $p_i$ are as small as allowed, or closest to the left
endpoints $t_i$ of the intervals $[t_i, t_i + \Delta t_i]$.
This implies
\begin{equation}\label{sum pi 2}
  \p( X_n \in B ) \ge  \frac{1}{n} \sum_{p_1,\dots,p_k}  \frac{n}{p_1\dots p_k} 
\ \rho\left( \frac{1 -t_1 - \dots
  - t_k}{t_k} \right) \ (1+o(1)),
\end{equation}
so that
\begin{equation}\label{sum pi 3}
  \p( X_n \in B ) \ge  \sum_{p_1}\frac{1}{p_1} \dots
  \sum_{p_k}\frac{1}{p_k}
\ \rho\left( \frac{1 -t_1 - \dots
  - t_k}{t_k} \right) \ (1+o(1)).
\end{equation}
Writing $[t,t+\Delta t]$ in place of $[t_i,t_i+\Delta t_i]$,
each of the $k$ sums in~\eqref{sum pi 3} has a positive limit,
derived 
from Mertens' theorem~\eqref{mertens}.
\begin{equation}\label{log factor}
\sum_{p \in [n^t,n^{t + \Delta t}]} \frac{1}{p} \  \to \log(t+\Delta
t) - \log t
\ = \log\left( 1 + \frac{\Delta t}{t} \right).
\end{equation}

At this point in the heuristic argument of Section~\ref{intro1}, we
simply replaced $\log\left( 1 + \Delta t/t \right)$ with $\Delta t/t$,
though for input to Proposition~\ref{proposition 1} an inequality of
the form $\log\left( 1 + \Delta t/t \right) \ge \Delta t/t$ would
easily suffice, if only it were true.  Unfortunately, for $r>0$, it is
the case that $\log(1+r)<r$; but we still have $\log(1+r)/r > 1-r/2$
for $r \in (0,1)$, and from this we will succeed in manufacturing a
$(1-\eps)$ lower bound on the product of $k$ factors of the form
$\log(1+r)/r$.
 
Proposition~\ref{proposition 1} was designed precisely to work in the
face of this weaker lower bound, and the uniformity requirement
\eqref{lemma hyp 0} can be met. Namely, let $R = k/(2 \eps)$.  We have
$\Delta t_i < \diam(B)$ and $d(B,U^c) \le t_1 \le t_i$, so $R \,\Delta
t_i < R \, \diam(B) < d(B, U^c) \le t_i$.  Hence the $r$ appearing in
$\log(1+r)$ satisfies $r = \Delta t_i / t_i < 1/R = 2 \eps/k$, so $r/2
< \eps/k$, and $(1-r/2)^k > (1-\eps/k)^k > 1-\eps$.

This, together with~\eqref{sum pi 3} and~\eqref{log factor}
shows that for a closed box $B \subset U$ satisfying 
\eqref{lemma hyp 0}
with $R = k/(2 \eps)$ we have
$$ 
\liminf_n \p( X_n \in B ) \ge (1-\eps)\ \frac{\Delta t_1}{t_1} \dots 
 \frac{\Delta t_k}{t_k} \ \rho\left( \frac{1 -t_1 - \dots
  - t_k}{t_k} \right) 
$$
$$ =  (1-\eps)\   \vol(B)\  f(t_1,\dots,t_k) 
\ \ge (1-\eps)\   \vol(B)\  \inf_B f
$$
so that~\eqref{lemma hyp 1} is satisfied with the required uniformity, and then Billingsley's
Poisson--Dirichlet convergence follows, by Proposition~\ref{proposition 1}. 

This completes the simple direct proof.
\end{proof}

\section{An Historical Survey}
\label{sect history}

In this section, we discuss the previously published proofs of
\eqref{billingsley}.
There are four different complete proofs, and also some partial proofs.

\subsection{Billingsley 1972}

Billingsley's original formulation in \cite{billingsley72}  
looks very different from present day versions of his result.
The Poisson--Dirichlet distribution had not yet appeared in published 
literature as a studied object with a name. Nor, in fact, is
Dickman's function $\rho$ mentioned explicitly in this 1972
paper, although de Bruijn (1951)~\cite{debruijn} is referenced.  Instead Billingsley introduces
functions $H_0,H_1,H_2,\dots$ on $(0,\infty)$, defined by $H_0(x) =1$ and
for $i \ge 1$,
$$
   H_i(u) := \int \prod_{k=1}^i \frac{dt_k}{t_k},
$$
where the integral is taken over the region
$$
   1 < t_1 < t_2 < \dots < t_i <x, \ \sum_{k=1}^i 1/t_k \ < 1.
$$
(So $H_i(x)=0$ if $x \le i$ since in that case the region is
empty.)  His limit result is then expressed in terms of these functions.

In modern notation, it is the case that, for $u>0$
\begin{equation}\label{billHsum}
   \rho(u)  = \sum_{i=0}^\infty (-1)^i H_i(u) 
= 1 + \sum_{1 \le i < u}(-1)^i H_i(u),
\end{equation}
though Billingsley does not seem to be aware of that formula, nor does
he refer to existing estimates on $\Psi(x,y)$.  Instead he applies
inclusion-exclusion directly to all the terms in the left-hand side of
\eqref{billingsley}, and arrives at formulas involving the right-hand
side of~\eqref{billHsum}.  In particular, his inclusion-exclusion
argument, specialized to the case $k=1$ in~\eqref{billingsley k},
shows (in modern notation) that $\Psi(x,x^{1/u}) \sim x \rho(u)$ for
each fixed $u>1$; and this special case $k=1$ might be viewed as a
rigorous version of Dickman's original argument.  The ``added value,''
then, of Billingsley 1972 \cite{billingsley72}, relative to Dickman
\cite{Dickman}, 1930, is the focus on the successively smaller prime
factors, as well as the derivation of their \emph{joint} distribution.

\subsection{Donnelly and Grimmett 1993} 

The paper of Donnelly and Grimmett \cite{DonnellyGrimmett} invokes a
size-biased permutation of the prime factors of a random integer.
This device exploits a (by then) known construction of the PD 
as the ranked list of values $1-U_1,U_1-U_1U_2,U_1U_2-U_1U_2U_3,\dots$
formed by independent $U_1,U_2,\dots$, uniformly distributed in (0,1).

Both size-biased permutations and the above expressions
involving uniform (0,1) variables are implicit in Eric Bach's 
1984 computer science dissertation \cite{bach}, in which he devises
and analyzes an efficient algorithm for the generation of large random
integers in factored form, solving a long-standing problem.  But there is no direct appearance of
Poisson--Dirichlet process in \cite{bach} nor, for that matter, any concern with
limit theorems. 
Donnelly and Grimmett, on their part, seem not to have known of this earlier work.

\subsection{Tenenbaum 2000}

Tenenbaum \cite{Tenenbaum2000},  which concerns the rate of convergence
in Billingsley's theorem, calculates 
a highly refined asymptotic series estimate of the difference
between the cumulative joint distribution function of the first $k$
coordinates of the PD, and the corresponding exact discrete
probabilities for the left-hand side of~\eqref{billingsley}.  
Billingsley's theorem is thus a corollary of Tenenbaum's result, 
though this line of argument is (necessarily) longer and more elaborate 
than other proofs of the limit result alone.  In hindsight,
one might say that the present proof replaces Tenenbaum's detailed hard estimates 
with easier estimates, plus the soft 
convergence Lemma~\ref{proposition 1}.

\subsection{Arratia 2002} 

Published in \cite{budalect}, this paper shows that for
$n=1,2,\dots$, the random integers $N(n)$, as in~\eqref{def L_i}, and
one copy of the Poisson--Dirichlet distribution, as on the right of
\eqref{billingsley}, can be constructed jointly so that
\begin{equation}\label{ell 1}
  \e \sum_{i \ge 1} \left| \frac{\log P_i(N(n))}{\log n} - L_i \right|
  = O\left( \frac{\log \log n}{\log n} \right).
\end{equation}
This formula proves~\eqref{billingsley}, and gives an upper bound on
the
expected $\ell_1$ distance.  It is conjectured that the expected
$\ell_1$ distance on the left side of~\eqref{ell 1} can be made as
small as $O(1/ \log n)$.

This paper is based on a size-biased permutation of the infinite
multiset of prime factors (each $p$ occurs with independent
multiplicity $Z_p$, geometrically distributed with $P(Z_p \ge k =
1/p^k$,) which makes it possible to couple prime counts with the
Poisson process ($dx/x$ on $(e^{-\gamma},\infty)$).  The infinite
size-biased permutation may be considered an extension of the
size-biased permutation used by Donnelly--Grimmett and Bach.

\subsection{Other Arguments}

Knuth and Trabb Pardo \cite{KnuthTrabbPardo} (1978), apparently
unaware of 
Billingsley's theorem though familiar with Dickman's work, derive
the limiting \emph{marginal} distributions of the individual $L_i(n)$'s in
terms of $\rho(\cdot)$. 

Vershik \cite{Vershik} 1986, apparently also unaware of Billingsley's
result, announced the PD limit result for prime factorizations.
But this paper supplies no proof, nor indication of method. 

Kingman, \cite{kingman}, who explicitly christened the Poisson--Dirichlet
distribution in \cite{kingman} (1975), has an 11 page preprint, ``The
Poisson--Dirichlet 
distribution and the frequency of large prime divisors,'' available at
{\tt www.newton.ac.uk/preprints/}\break{\tt NI04019.pdf}.  This preprint gives
an analog of Billingsley's theorem, in which harmonic density is
substituted for natural density.  Kingman cites \cite{lloyd} for
inspiration, and \cite{levinfainleib} for providing techniques to show
the existence of a certain natural density; the latter, combined with
Kingman's result, would constitute yet another full proof of
Billingsley's theorem.

\bibliography{direct}
\bibliographystyle{plain}

\end{document}